\documentclass[11pt]{article}

\usepackage[T1]{fontenc}
\usepackage[utf8]{inputenc}
\usepackage[a4paper,margin=1in]{geometry}
\usepackage{graphicx}
\usepackage{multirow}
\usepackage{amsmath,amssymb,amsfonts}
\usepackage{amsthm}
\usepackage[title]{appendix}
\usepackage{xcolor}
\usepackage{textcomp}
\usepackage{manyfoot}
\usepackage{booktabs}
\usepackage{algorithm}
\usepackage{algorithmicx}
\usepackage{algpseudocode}
\usepackage{listings}
\usepackage{braket}
\usepackage[colorlinks=true,linkcolor=blue,citecolor=blue,urlcolor=blue]{hyperref}

\theoremstyle{plain}
\newtheorem{theorem}{Theorem}
\newtheorem{proposition}[theorem]{Proposition}
\newtheorem{corollary}[theorem]{Corollary}
\newtheorem{lemma}[theorem]{Lemma}

\theoremstyle{definition}

\theoremstyle{remark}
\newtheorem{remark}[theorem]{Remark}

\newcommand{\Hh}{\mathcal{H}}
\DeclareMathOperator{\Tr}{Tr}
\newcommand{\rhodiag}{\rho_{\mathrm{diag}}}
\newcommand{\Wavg}{\mathcal{W}}
\newcommand{\R}{\mathbb{R}}
\DeclareMathOperator{\dist}{dist}
\newcommand{\C}{\mathbb{C}}
\newcommand{\Z}{\mathbb{Z}}
\newcommand{\Aeq}{A_{\mathrm{eq}}}

\newcommand{\safeincludegraphics}[2][]{%
  \IfFileExists{#2}{\includegraphics[#1]{#2}}{%
    \IfFileExists{#2.pdf}{\includegraphics[#1]{#2.pdf}}{%
      \IfFileExists{#2.png}{\includegraphics[#1]{#2.png}}{%
        \IfFileExists{#2.jpg}{\includegraphics[#1]{#2.jpg}}{%
          \IfFileExists{#2.jpeg}{\includegraphics[#1]{#2.jpeg}}{%
            \fbox{\parbox[c][0.25\textheight][c]{0.8\linewidth}{\centering Missing figure: \texttt{\detokenize{#2}}}}%
          }%
        }%
      }%
    }%
  }%
}

\title{Weighted Time Averages and Weak Convergence to Equilibrium in Quantum Integrable Systems}
\author{%
Xinyu Liu\textsuperscript{1} \and
Yong Li\textsuperscript{1,2}\thanks{Corresponding author. Email: \texttt{liyong@jlu.edu.cn}}}
\date{\today}

\begin{document}
\maketitle

\begin{center}
\textsuperscript{1}\,College of Mathematics, Jilin University, Changchun 130012, P. R. China\\[2pt]
\textsuperscript{2}\,Center for Mathematics and Interdisciplinary Sciences, Northeast Normal University, Changchun 130024, P. R. China
\end{center}

\begin{abstract}
This paper investigates the quantum counterpart of the weak convergence of statistical ensembles to equilibrium in classical integrable systems. In contrast to the classical case, one-time observables in quantum integrable systems with pure point spectrum typically exhibit quasiperiodic oscillations and therefore do not admit pointwise long-time limits. To overcome this difficulty, we introduce a weighted time-averaging approach and identify the diagonal (dephased) state as the natural equilibrium object in the quantum setting, thereby establishing a corresponding convergence theory. Moreover, for finite-frequency quasiperiodic signals, we derive explicit quantitative convergence estimates and reveal the acceleration mechanism of weighted averages relative to ordinary time averages. Finally, we construct an explicitly solvable quantum integrable model and perform numerical simulations that confirm the theoretical results and demonstrate the effectiveness of weighted time averages in approximating the equilibrium state.
\end{abstract}


\section{Introduction}\label{sec1}

The long-time equilibration and thermalization of isolated quantum systems is one of the central problems in nonequilibrium statistical physics, quantum dynamics, and mathematical physics \cite{Polkovnikov2011RMP,Eisert2015NatPhys,GogolinEisert2016,DAlessio2016AdvPhys}. Over the past two decades, the rapid development of quantum simulation platforms---including ultracold atoms, low-dimensional quantum gases, and superconducting qubits---has made it possible to probe the nonequilibrium dynamics of closed many-body systems directly under highly isolated and strongly controllable experimental conditions, thereby greatly stimulating the corresponding theoretical advances \cite{Eisert2015NatPhys,Langen2015,Kaufman2016,Rauer2018}. For generic non-integrable systems, the long-time behavior is typically closely tied to quantum chaos and the eigenstate thermalization mechanism; a large body of work shows that, under suitable conditions, local observables relax to thermal equilibrium values predicted by equilibrium statistical mechanics \cite{Rigol2008Nature,DAlessio2016AdvPhys,GogolinEisert2016}.

The situation is fundamentally different for integrable systems. Due to the presence of extensively many conserved quantities, their long-time behavior is generally no longer described by the standard Gibbs ensemble, but is instead more naturally characterized by generalized Gibbs ensembles, diagonal ensembles, or, more broadly, dephasing mechanisms \cite{Rigol2007PRL,Cassidy2011,CauxEssler2013,Ilievski2015,Langen2015}. This suggests that, although a single pure state in an integrable quantum system still undergoes reversible unitary evolution, the long-time behavior of suitable observables may nevertheless exhibit robust statistical structure. On the classical side, Mitchell established in 2019 a weak convergence result to equilibrium for statistical ensembles in integrable Hamiltonian systems \cite{Mitchell2019}. Building on this perspective, the ensemble-theoretic study of integrable dynamics was subsequently extended to discrete-time settings, limit theorems, and non-autonomous transitions, including the law of large numbers and the central limit theorem for discrete integrable Hamiltonian systems, as well as long-time averaging and weak convergence for integrable Hamiltonian systems with almost periodic transitions \cite{LiuZhangLi2026,LiuLi2025}.

From a broader viewpoint, however, being ``close to equilibrium'' in an isolated quantum system does not necessarily mean that the system converges strongly to a stationary limiting state. More commonly, certain observables become stable only in a time-averaged or weak sense \cite{Reimann2008,Cramer2008,Barthel2008,GogolinEisert2016}. This point is especially important for quantum systems with pure point spectrum: because their time evolution is built from oscillatory phases associated with discrete energy gaps, single-time expectation values typically display quasi-periodic or almost periodic behavior, and one therefore should not expect pointwise convergence as \(t\to\infty\). Experimental observations of generalized equilibrium, thermalization, and nontrivial recurrence phenomena further indicate that long-time behavior often contains two simultaneous features, namely local stabilization and persistent oscillations or revivals \cite{Langen2015,Kaufman2016,Rauer2018}. Accordingly, in the pure-point setting, the genuinely natural asymptotic object is not the pointwise limit of a single-time signal, but rather an effective equilibrium state extracted through an appropriate time-averaging procedure.

Motivated by this viewpoint, in the quantum setting we seek a notion of weak equilibration that is compatible with spectral decomposition. For the pure-point systems considered in this work, the most natural candidate is precisely the diagonal (or dephased) state \(\rhodiag\) induced by the initial state through the energy decomposition. This state removes the oscillatory couplings between distinct energy subspaces while retaining the block structure inside degenerate eigenspaces, and therefore captures exactly the information that survives after long-time averaging.

Identifying the limiting object alone, however, is not sufficient. In the integrable setting, the truly meaningful questions concern not only what the equilibrium object is, but also by what mechanism and at what rate it is approached. In this regard, the recent series of works by Tong and Li \cite{TongLi2024,TongLi2025,TongLi2026} is of particular importance. By developing a systematic theory of weighted Birkhoff averages based on endpoint-flat weight functions, they have transformed time-averaging problems from questions of existence and limit identification into a framework that also permits refined quantitative analysis. The present work pushes this weighted averaging theory into the setting of quantum integrable systems. We introduce a class of compactly supported, \(C^\infty\) weight functions that decay rapidly near both endpoints of the averaging interval, and use them to construct continuous and discrete weighted time averages. The significance of this construction lies not merely in improving numerical efficiency. At the methodological level, it builds a bridge between fast time-averaging theory and the problems of dephasing and weak equilibration in quantum integrable systems. As a consequence, the present paper not only identifies the relevant equilibrium object in the quantum integrable setting, but also establishes explicit quantitative convergence rates in an appropriate finite-frequency framework.

In summary, this work develops weak convergence results for statistical ensembles associated with integrable quantum Hamiltonian systems and proposes a quantum weak-equilibration framework in terms of weighted dephasing toward the diagonal equilibrium state. By bringing weighted Birkhoff averaging into quantum integrable dynamics, we obtain both general convergence results and explicit convergence rates, with quantitative dependence on the parameters \((p,q)\), in the finite-frequency case. In this way, the idea of weak equilibration is unified with the theory of rapidly convergent weighted averages in the quantum setting. In addition, we construct a completely solvable three-spin quantum integrable model, allowing a direct comparison between the theoretical results and numerical simulations. This example clearly demonstrates the acceleration effect of weighted averages relative to ordinary averages and, under suitable model choices and high-precision numerical implementation, confirms the faster convergence associated with larger values of \(\min\{p,q\}\). It should be emphasized, however, that, as also stressed by Tong and Li \cite{TongLi2026}, the effective convergence observed at finite sample size is influenced not only by the asymptotic theory, but also by normalization constants, nonresonance structure, and numerical errors. Accordingly, the numerical results in this paper should be understood as strong evidence for the underlying mechanism, rather than as a naive universal prediction valid for all models.

The remainder of the paper is organized as follows. In Section~2, we introduce quantum systems with pure point spectrum, the diagonal (dephased) state, and the basic definitions of weighted time averages. Section~3 establishes a general weighted dephasing theorem, the weak convergence of weighted averaged states, and the quantitative acceleration estimates in the finite-frequency case. In Section~4, we study an explicit three-spin quantum integrable model and verify the theoretical results through analytic calculations and numerical experiments. Finally, Section~5 concludes the paper and discusses several possible directions for future research.

\section{Preliminaries}\label{sec2}
In this section, we describe the quantum system studied in this paper, formulate the key definitions, and introduce the notation needed in the subsequent analysis.

\subsection{Quantum systems with pure point spectrum}

Let \(\Hh\) be a separable Hilbert space, endowed with inner product \(\langle \cdot,\cdot\rangle\).
We consider a self-adjoint operator
\begin{equation*}
H:D(H)\subset \Hh \to \Hh,
\end{equation*}
which represents the Hamiltonian of the system.

Throughout this paper, we assume that \(H\) has \emph{pure point spectrum}. This means that the spectrum of \(H\) consists entirely of eigenvalues, and that \(\Hh\) is spanned by the eigenvectors of \(H\); equivalently, \(H\) has no continuous spectral component.
By the spectral theorem, there exists at most a countable family of eigenvalues \(\lambda\in\sigma(H)\), together with the corresponding orthogonal spectral projections \(\Pi_\lambda\), such that
\begin{equation}\label{eq:pure-point-decomp}
H=\sum_{\lambda\in\sigma(H)} \lambda\,\Pi_\lambda.
\end{equation}
If \(H\) is unbounded, the above expansion is understood to hold on the domain \(D(H)\) in the strong operator topology; if \(H\) is bounded, it may be interpreted directly in the usual operator sense.

Moreover, \(\Pi_\lambda\) denotes the \emph{spectral projection} associated with the eigenvalue \(\lambda\), namely, the orthogonal projection from \(\Hh\) onto the eigenspace of energy \(\lambda\),
\begin{equation*}
E_\lambda:=\ker(H-\lambda I).
\end{equation*}
In other words, for any \(\psi\in\Hh\), the vector \(\Pi_\lambda\psi\) is precisely the component of \(\psi\) lying in the energy subspace \(E_\lambda\).
These spectral projections satisfy
\begin{equation*}
\Pi_\lambda^\ast=\Pi_\lambda,\qquad
\Pi_\lambda^2=\Pi_\lambda,\qquad
\Pi_\lambda\Pi_\mu=0\quad (\lambda\neq\mu),
\end{equation*}
and moreover,
\begin{equation*}
\sum_{\lambda\in\sigma(H)} \Pi_\lambda=I.
\end{equation*}

Therefore, the Hilbert space admits the orthogonal direct-sum decomposition
\begin{equation*}
\Hh=\bigoplus_{\lambda\in\sigma(H)} E_\lambda.
\end{equation*}
In the above decomposition, if the eigenspace corresponding to an eigenvalue \(\lambda\) satisfies
\[
\dim E_\lambda = 1,
\]
then \(\lambda\) is said to be \emph{non-degenerate}. If, on the other hand,
\[
\dim E_\lambda > 1,
\]
then \(\lambda\) is said to be \emph{degenerate}, and one also says that the system exhibits \emph{spectral degeneracy} at the energy level \(\lambda\).
Physically, this means that multiple linearly independent quantum states correspond to the same energy level.

Spectral degeneracy is essential in this paper, since long-time averaging or dephasing removes the off-diagonal couplings \emph{between distinct energy subspaces}, while retaining the block structure \emph{inside each degenerate energy subspace}.

Let \(\rho_0\) denote the initial state of the system. In this paper, \(\rho_0\) is taken to be a density operator, that is, a trace-class operator satisfying
\[
\rho_0\ge 0,\qquad \Tr(\rho_0)=1.
\]
This allows the system to be in a general statistical mixed state.

In particular, if there exists a normalized vector
\[
\psi_0\in \Hh,\qquad \|\psi_0\|=1,
\]
such that
\begin{equation}\label{eq:pure-state-rho0}
\rho_0=\ket{\psi_0}\bra{\psi_0},
\end{equation}
then the system is said to be in a \emph{pure state}.
Here \(\psi_0\) is the state vector of the system at the initial time, and it completely characterizes the corresponding pure quantum state.
In Dirac notation, \(\ket{\psi_0}\bra{\psi_0}\) denotes the rank-one orthogonal projection generated by \(\psi_0\); more precisely, for any \(\varphi\in\Hh\),
\[
(\ket{\psi_0}\bra{\psi_0})\varphi
=
\braket{\psi_0|\varphi}\,\ket{\psi_0}.
\]
Thus, \(\rho_0=\ket{\psi_0}\bra{\psi_0}\) means that the system is in the state \(\psi_0\) with probability \(1\).

More generally, if
\[
\rho_0=\sum_j p_j \ket{\psi_j}\bra{\psi_j},
\qquad
p_j\ge 0,\quad \sum_j p_j=1,
\]
then \(\rho_0\) represents a statistical mixture of the pure states \(\psi_j\).

Under Hamiltonian evolution, the quantum state at time \(t\) is given by
\begin{equation}\label{eq:rho-t}
\rho_t=e^{-itH}\rho_0 e^{itH}.
\end{equation}
Here \(e^{-itH}\) denotes the unitary group generated by the self-adjoint operator \(H\).

If the initial state is pure, namely, \(\rho_0=\ket{\psi_0}\bra{\psi_0}\), then the corresponding state vector evolves according to the Schr\"odinger dynamics
\begin{equation}\label{eq:psi-t}
\psi_t=e^{-itH}\psi_0,
\end{equation}
and one has
\[
\rho_t=\ket{\psi_t}\bra{\psi_t}.
\]
Therefore, in the pure-state case, \(\psi_0\) is the most fundamental initial datum, while the density operator \(\rho_0\) is simply its equivalent operator representation. In the mixed-state case, however, one must work with \(\rho_0\) in order to describe the statistical state of the system in a unified manner.

For any bounded self-adjoint operator \(A\in\mathcal B(\Hh)\), we call \(A\) an observable, and define its expectation value in the state \(\rho_t\) by
\begin{equation}\label{eq:expectation-A}
\langle A\rangle_t:=\Tr(\rho_t A).
\end{equation}
Here \(\mathcal B(\Hh)\) denotes the algebra of all bounded linear operators on \(\Hh\), and \(\Tr\) stands for the operator trace.

Corresponding to the angle-averaged equilibrium state in classical integrable systems, a natural ``equilibrium object'' for quantum systems with pure point spectrum is the \emph{diagonal (dephased) state}
\begin{equation}\label{eq:diag-state}
\rhodiag:=\sum_{\lambda\in\sigma(H)} \Pi_\lambda \rho_0 \Pi_\lambda.
\end{equation}
Since \(\rho_0\) is a trace-class operator and each \(\Pi_\lambda\) is a bounded orthogonal projection, the above expression is well defined in the trace-norm sense.

From a structural point of view, \(\rhodiag\) is precisely the block-diagonal part of \(\rho_0\) with respect to the energy decomposition
\[
\Hh=\bigoplus_{\lambda\in\sigma(H)} E_\lambda.
\]
All off-diagonal components connecting distinct energy subspaces \(E_\lambda\) and \(E_\mu\) (\(\lambda\neq\mu\)) are removed, whereas the part within each fixed energy block \(E_\lambda\) is retained.

If the spectrum is completely non-degenerate, that is, if each \(E_\lambda\) is one-dimensional, and if \(\{e_n\}\) is a normalized eigenbasis such that
\[
He_n=E_n e_n,
\]
then
\[
\Pi_{E_n}=\ket{e_n}\bra{e_n},
\]
and hence
\[
\rhodiag
=
\sum_n \ket{e_n}\bra{e_n}\,\rho_0\,\ket{e_n}\bra{e_n}.
\]
In this case, \(\rhodiag\) reduces to the usual diagonal ensemble, namely, the operator obtained by retaining only the genuinely diagonal entries in the energy eigenbasis.

If the spectrum is degenerate, then \(\rhodiag\) is generally not fully diagonal entry by entry; rather, it is only \emph{block-diagonal} with respect to the energy decomposition.
This point is crucial, because under time evolution phase factors of the form
\[
e^{-it(\lambda-\mu)}
\]
arise between different energy sectors.
When \(\lambda\neq\mu\), these phases oscillate and are therefore washed out by time averaging; when \(\lambda=\mu\), however, the factor is identically equal to \(1\), so the terms inside the same degenerate energy block are not averaged out.
For this reason, the state \(\rhodiag\) defined in \eqref{eq:diag-state} is the most natural, and also the most robust, notion of equilibrium state for quantum systems with pure point spectrum.

\subsection{Weight Functions and Weak Convergence of Weighted Averaged States}
In this paper we adopt the same weight function \(w_{p,q}(x)\) as that used in Tong-Li \cite{TongLi2026}. 
Fix \(p,q>0\). Define
\begin{equation}
w_{p,q}(x)=
\begin{cases}
C_{p,q}^{-1}\exp\!\bigl(-x^{-p}(1-x)^{-q}\bigr), & x\in(0,1),\\[4pt]
0, & x\notin(0,1),
\end{cases}
\label{eq:weight}
\end{equation}
where
\[
C_{p,q}:=\int_0^1 \exp\!\bigl(-s^{-p}(1-s)^{-q}\bigr)\,ds.
\]
Then \(w_{p,q}\in C_0^\infty([0,1])\), and it satisfies
\[
\int_0^1 w_{p,q}(x)\,dx=1.
\]
For a scalar time-dependent function \(g(t)\), we define its continuous weighted average by
\begin{equation}\label{eq:cont-weighted}
\Wavg_T^{(p,q)}[g]
:=
\frac{\int_0^T w_{p,q}(t/T)\,g(t)\,dt}
{\int_0^T w_{p,q}(t/T)\,dt}
=
\int_0^1 w_{p,q}(s)\,g(Ts)\,ds.
\end{equation}
For a sequence \(\{g_n\}_{n\ge 0}\), we define the discrete weighted average by
\begin{equation}\label{eq:disc-weighted}
W_N^{(p,q)}(g)
:=
\frac{\sum_{n=0}^{N-1} w_{p,q}(n/N)\,g_n}
{\sum_{n=0}^{N-1} w_{p,q}(n/N)}.
\end{equation}

\section{Main Results}\label{sec3}
In this section, we present the main theoretical results of the paper and clarify the logical hierarchy among them. 
First, within the general framework of quantum systems with pure point spectrum, we prove that the weighted time average of any observable satisfying a suitable absolute summability condition converges to its expectation value in the diagonal (dephased) equilibrium state. This result establishes, at an abstract level, the basic mechanism of ``weighted weak equilibration'' in quantum integrable systems. 
Next, in the case of quasiperiodic signals with finitely many frequencies, we further derive explicit quantitative convergence estimates. In particular, we show that weighted averages exhibit a stretched-exponential acceleration compared with ordinary averages, with the exponent determined by the weight parameters \((p,q)\). 
These general conclusions provide the direct theoretical foundation for the analytic computations and numerical verification carried out in the explicit quantum integrable model studied in the next section.

\subsection{A General Weighted Dephasing Theorem}

We begin by stating a general convergence theorem.

\begin{theorem}\label{thm:main-dephase}
Let \(H\) be a self-adjoint operator with pure point spectrum, let \(\rho_0\) be an initial density operator, and let \(A\in \mathcal B(\Hh)\) be a bounded observable. 
Let \(\{e_m\}_{m\ge 1}\) be an orthonormal eigenbasis of \(H\), satisfying
\[
He_m=E_m e_m.
\]
Assume that
\begin{equation}
\sum_{m,n\ge 1} |\rho_{mn}A_{nm}|<\infty,
\qquad
\rho_{mn}:=\langle e_m,\rho_0 e_n\rangle,\quad
A_{nm}:=\langle e_n,A e_m\rangle.
\label{eq:abs-sum}
\end{equation}
Then, for any \(p,q>0\),
\begin{equation}
\lim_{T\to\infty}\Wavg_T^{(p,q)}[\langle A\rangle_\cdot]
=
\Tr(\rhodiag A).
\label{eq:main-limit}
\end{equation}
\end{theorem}

\begin{proof}
By the spectral decomposition, one has
\[
\rho_t
=
\sum_{m,n\ge 1}\rho_{mn} e^{-it(E_m-E_n)}\ket{e_m}\bra{e_n},
\]
and hence
\begin{equation}
\langle A\rangle_t
=
\Tr(\rho_t A)
=
\sum_{m,n\ge 1}\rho_{mn}A_{nm}e^{-it(E_m-E_n)}.
\label{eq:expansion-A}
\end{equation}
By assumption \eqref{eq:abs-sum}, the series on the right-hand side is absolutely summable, so the order of summation and the subsequent integration may be interchanged. Substituting \eqref{eq:expansion-A} into the continuous weighted average \eqref{eq:cont-weighted}, we obtain
\[
\Wavg_T^{(p,q)}[\langle A\rangle_\cdot]
=
\sum_{m,n\ge 1}\rho_{mn}A_{nm}
\int_0^1 w_{p,q}(s)e^{-iT(E_m-E_n)s}\,ds.
\]
Define
\[
F_T(\omega):=\int_0^1 w_{p,q}(s)e^{-iT\omega s}\,ds.
\]
Then \(F_T(0)=1\), since \(\int_0^1 w_{p,q}(s)\,ds=1\). If \(\omega\neq 0\), then, because \(w_{p,q}\in L^1(0,1)\), the Riemann--Lebesgue lemma \cite{Mitchell2019} yields
\[
F_T(\omega)\to 0
\qquad (T\to\infty).
\]
Therefore, for each pair \((m,n)\),
\[
\rho_{mn}A_{nm}F_T(E_m-E_n)
\to
\begin{cases}
\rho_{mn}A_{nm}, & E_m=E_n,\\[4pt]
0, & E_m\neq E_n.
\end{cases}
\]
Moreover,
\[
|F_T(\omega)|
\le \int_0^1 w_{p,q}(s)\,ds=1,
\]
so each summand is dominated by the absolutely summable family \(|\rho_{mn}A_{nm}|\). By dominated convergence for series, it follows that
\[
\lim_{T\to\infty}\Wavg_T^{(p,q)}[\langle A\rangle_\cdot]
=
\sum_{E_m=E_n}\rho_{mn}A_{nm}.
\]
On the other hand,
\[
\rhodiag=\sum_\lambda \Pi_\lambda \rho_0 \Pi_\lambda
\]
retains precisely those matrix elements belonging to the blocks with \(E_m=E_n\) in the eigenbasis. Hence
\[
\sum_{E_m=E_n}\rho_{mn}A_{nm}
=
\Tr(\rhodiag A).
\]

This proves the theorem.
\end{proof}

\begin{corollary}
\label{cor:weak-state}
Let \(\mathfrak A\subset \mathcal B(\Hh)\) be a class of observables such that, for every \(A\in\mathfrak A\), the absolute summability condition \eqref{eq:abs-sum} in Theorem~\ref{thm:main-dephase} is satisfied. Define
\begin{equation}
\bar\rho_T^{(p,q)}
:=
\frac{\int_0^T w_{p,q}(t/T)\,\rho_t\,dt}
{\int_0^T w_{p,q}(t/T)\,dt}.
\label{eq:avg-state-cor}
\end{equation}
Then, as \(T\to\infty\),
\[
\bar\rho_T^{(p,q)} \rightharpoonup \rhodiag
\]
weakly on \(\mathfrak A\); that is, for every \(A\in\mathfrak A\),
\[
\Tr\!\bigl((\bar\rho_T^{(p,q)}-\rhodiag)A\bigr)\to 0.
\]
\end{corollary}

\begin{proof}
Set
\[
Z_T:=\int_0^T w_{p,q}(t/T)\,dt.
\]
Since \(w_{p,q}\ge 0\) and is strictly positive on \((0,1)\), it follows that \(Z_T>0\) for every \(T>0\).

On the other hand, since \(\rho_0\) is a density operator, it is trace class and satisfies
\[
\|\rho_0\|_1=\Tr(\rho_0)=1.
\]
Moreover,
\[
\rho_t=e^{-itH}\rho_0 e^{itH},
\]
and unitary conjugation preserves the trace norm. Hence, for every \(t\in\mathbb R\),
\[
\|\rho_t\|_1=\|\rho_0\|_1=1.
\]

We next show that the map \(t\mapsto \rho_t\) is continuous in the trace norm. We first verify this for finite-rank operators. If
\[
F=\sum_{j=1}^N \ket{\phi_j}\bra{\psi_j},
\]
then
\[
e^{-itH}F e^{itH}
=
\sum_{j=1}^N \ket{e^{-itH}\phi_j}\bra{e^{-itH}\psi_j}.
\]
Since \(e^{-itH}\) is strongly continuous on \(\Hh\) with respect to \(t\), the above expression depends continuously on \(t\) in the trace norm for every finite-rank operator \(F\). Using the density of finite-rank operators in the trace-class space \(\mathfrak S_1(\Hh)\), together with the fact that unitary conjugation preserves the trace norm, we conclude that \(t\mapsto \rho_t\) is trace-norm continuous for general trace-class \(\rho_0\) as well.

Therefore,
\[
t\longmapsto w_{p,q}(t/T)\rho_t
\]
is a Bochner integrable \(\mathfrak S_1(\Hh)\)-valued function on \([0,T]\). Hence
\[
\int_0^T w_{p,q}(t/T)\rho_t\,dt
\]
is well defined in the trace-class sense, and thus \(\bar\rho_T^{(p,q)}\) is well defined.

Now fix any \(A\in \mathfrak A\subset \mathcal B(\Hh)\). Consider the linear functional on \(\mathfrak S_1(\Hh)\) defined by
\[
L_A(X):=\Tr(XA).
\]
Since \(A\) is bounded, \(L_A\) is a continuous linear functional on \(\mathfrak S_1(\Hh)\), and satisfies
\[
|L_A(X)|=|\Tr(XA)|\le \|A\|\,\|X\|_1.
\]
By the compatibility of the Bochner integral with continuous linear functionals,
\[
L_A\!\left(\int_0^T w_{p,q}(t/T)\rho_t\,dt\right)
=
\int_0^T w_{p,q}(t/T)L_A(\rho_t)\,dt.
\]
Therefore,
\[
\Tr\!\left(\bar\rho_T^{(p,q)}A\right)
=
\frac{1}{Z_T}
\int_0^T w_{p,q}(t/T)\Tr(\rho_tA)\,dt.
\]
Recalling that \(\langle A\rangle_t=\Tr(\rho_tA)\), we obtain
\[
\Tr\!\left(\bar\rho_T^{(p,q)}A\right)
=
\frac{\int_0^T w_{p,q}(t/T)\,\langle A\rangle_t\,dt}
{\int_0^T w_{p,q}(t/T)\,dt}
=
\Wavg_T^{(p,q)}[\langle A\rangle_\cdot].
\]

Since \(A\in\mathfrak A\), the hypothesis of the corollary ensures that \(A\) satisfies the absolute summability condition \eqref{eq:abs-sum} in Theorem~\ref{thm:main-dephase}. We may therefore apply Theorem~\ref{thm:main-dephase} to obtain
\[
\lim_{T\to\infty}\Wavg_T^{(p,q)}[\langle A\rangle_\cdot]
=
\Tr(\rhodiag A).
\]
Combining this with the identity proved above yields
\[
\lim_{T\to\infty}\Tr\!\left(\bar\rho_T^{(p,q)}A\right)
=
\Tr(\rhodiag A).
\]
Hence
\[
\lim_{T\to\infty}\Tr\!\bigl((\bar\rho_T^{(p,q)}-\rhodiag)A\bigr)=0.
\]

By the definition of weak convergence on \(\mathfrak A\), this is exactly the statement that
\[
\bar\rho_T^{(p,q)} \rightharpoonup \rhodiag
\qquad (T\to\infty)
\]
weakly on \(\mathfrak A\).
This completes the proof.
\end{proof}

\subsection{Quantitative Acceleration for Quasiperiodic Signals with Finitely Many Frequencies}

Theorem~\ref{thm:main-dephase} establishes the limiting behavior, but does not provide a convergence rate. To obtain quantitative results, we now turn to quasiperiodic signals with finitely many frequencies, which also correspond to the setting of the numerical model considered later in the paper.

\begin{lemma}\label{lem:fourier-decay}
Regard \(w_{p,q}\) as a function defined on the whole real line \(\mathbb R\), extended by zero outside \((0,1)\). Then there exist constants \(c,C>0\), depending only on \(p\) and \(q\), such that
\begin{equation}\label{eq:fourier-decay}
|\widehat w_{p,q}(\xi)|
\le
C\exp\!\bigl(-c|\xi|^{\zeta(p,q)}\bigr),
\qquad
\zeta(p,q)=\left(1+\frac{1}{\min\{p,q\}}\right)^{-1}.
\end{equation}
\end{lemma}

\begin{proof}
Set
\[
r:=\min\{p,q\},
\qquad
s:=1+\frac1r.
\]
Note that
\[
\zeta(p,q)=\frac1s.
\]
Throughout this paper, we use the Fourier transform convention
\[
\widehat w_{p,q}(\xi)
:=
\int_{\R} e^{-i\xi x} w_{p,q}(x)\,dx.
\]

We first prove that there exist constants \(A_0,B_0>0\) such that, for every integer \(m\ge 0\),
\begin{equation}
\|w_{p,q}^{(m)}\|_{L^1(\R)}
\le
A_0 B_0^m (m!)^s.
\label{eq:L1-Gevrey}
\end{equation}

Since the support of \(w_{p,q}\) is contained in \([0,1]\), it suffices to estimate its derivatives on the interval \([0,1]\).
We split \([0,1]\) into the two subintervals
\[
[0,1]=\Bigl[0,\frac12\Bigr]\cup\Bigl[\frac12,1\Bigr].
\]

When \(x\in[0,\frac12]\), we may write
\[
w_{p,q}(x)
=
C_{p,q}^{-1}\exp\!\bigl(-a_q(x)x^{-p}\bigr),
\qquad
a_q(x):=(1-x)^{-q}.
\]
The function \(a_q\) is analytic in a complex neighborhood of \([0,\frac12]\), and there exist constants \(0<a_- \le a_+<\infty\) such that
\[
a_- \le a_q(x)\le a_+,
\qquad x\in\Bigl[0,\frac12\Bigr].
\]
Therefore, near the left endpoint, the structure of \(w_{p,q}\) is completely analogous to that of the standard endpoint-flat function
\[
x\longmapsto e^{-x^{-p}},
\]
up to multiplication by a positive analytic coefficient \(a_q(x)\).
By the standard Gevrey estimate for endpoint-flat functions, there exist constants \(A_1,B_1>0\) such that, for all \(m\ge 0\),
\begin{equation}
\sup_{x\in[0,1/2]} |w_{p,q}^{(m)}(x)|
\le
A_1 B_1^m (m!)^{1+1/p}.
\label{eq:left-Gevrey}
\end{equation}

Similarly, when \(x\in[\frac12,1]\), we may write
\[
w_{p,q}(x)
=
C_{p,q}^{-1}\exp\!\bigl(-b_p(x)(1-x)^{-q}\bigr),
\qquad
b_p(x):=x^{-p}.
\]
The function \(b_p\) is analytic in a complex neighborhood of \([\frac12,1]\), and there exist constants \(0<b_- \le b_+<\infty\) such that
\[
b_- \le b_p(x)\le b_+,
\qquad x\in\Bigl[\frac12,1\Bigr].
\]
By the same endpoint-flat Gevrey estimate, there exist constants \(A_2,B_2>0\) such that, for all \(m\ge 0\),
\begin{equation}
\sup_{x\in[1/2,1]} |w_{p,q}^{(m)}(x)|
\le
A_2 B_2^m (m!)^{1+1/q}.
\label{eq:right-Gevrey}
\end{equation}

Since
\[
s=1+\frac1r=1+\frac1{\min\{p,q\}}
=\max\left\{1+\frac1p,\ 1+\frac1q\right\},
\]
we have
\[
(m!)^{1+1/p}\le (m!)^s,
\qquad
(m!)^{1+1/q}\le (m!)^s.
\]
Integrating \eqref{eq:left-Gevrey} and \eqref{eq:right-Gevrey} over finite intervals, and enlarging the constants if necessary, we obtain constants \(A_0,B_0>0\) such that
\[
\|w_{p,q}^{(m)}\|_{L^1(\R)}
=
\|w_{p,q}^{(m)}\|_{L^1(0,1)}
\le
A_0 B_0^m (m!)^s,
\]
which proves \eqref{eq:L1-Gevrey}.

Since \(w_{p,q}\in C_0^\infty(\R)\), for any \(\xi\neq 0\) and any integer \(m\ge 0\), repeated integration by parts in the Fourier integral yields
\[
\widehat w_{p,q}(\xi)
=
(i\xi)^{-m}
\int_{\R} e^{-i\xi x} w_{p,q}^{(m)}(x)\,dx.
\]
Hence
\begin{equation}
|\widehat w_{p,q}(\xi)|
\le
|\xi|^{-m}\,\|w_{p,q}^{(m)}\|_{L^1(\R)}.
\label{eq:ibp-fourier}
\end{equation}
Substituting \eqref{eq:L1-Gevrey} into \eqref{eq:ibp-fourier}, we obtain
\begin{equation}
|\widehat w_{p,q}(\xi)|
\le
A_0\left(\frac{B_0}{|\xi|}\right)^m (m!)^s.
\label{eq:pre-optimize}
\end{equation}

We now optimize with respect to \(m\). Using the crude estimate \(m!\le m^m\) for \(m\ge 1\), we get
\[
(m!)^s\le m^{sm}.
\]
Therefore,
\begin{equation}
|\widehat w_{p,q}(\xi)|
\le
A_0\left(\frac{B_0 m^s}{|\xi|}\right)^m.
\label{eq:rough-optimize}
\end{equation}

Now assume that \(|\xi|\) is sufficiently large, and choose
\begin{equation}
m
=
\left\lfloor
\left(\frac{|\xi|}{2B_0}\right)^{1/s}
\right\rfloor.
\label{eq:m-choice}
\end{equation}
For \(|\xi|\) large enough, we have \(m\ge 1\), and since \(\lfloor x\rfloor \le x\),
\[
m^s \le \frac{|\xi|}{2B_0}.
\]
Substituting this into \eqref{eq:rough-optimize}, we obtain
\[
|\widehat w_{p,q}(\xi)|
\le
A_0\,2^{-m}.
\]
On the other hand, for sufficiently large \(|\xi|\), the bound \(\lfloor x\rfloor \ge x/2\) for \(x\ge 1\) implies that there exists a constant \(c_1>0\) such that
\[
m\ge c_1 |\xi|^{1/s}.
\]
Hence
\[
2^{-m}
\le
\exp\!\bigl(-(\log 2)\,m\bigr)
\le
\exp\!\bigl(-c_2 |\xi|^{1/s}\bigr)
\]
for some constant \(c_2>0\). Therefore, for all sufficiently large \(|\xi|\),
\[
|\widehat w_{p,q}(\xi)|
\le
A_0 \exp\!\bigl(-c_2 |\xi|^{1/s}\bigr).
\]

For \(\xi\) in a bounded interval, it suffices to use the fact that \(w_{p,q}\in L^1(\R)\), which gives
\[
|\widehat w_{p,q}(\xi)|
\le
\|w_{p,q}\|_{L^1(\R)}.
\]
Hence, after enlarging the constant \(C\) if necessary, the small- and large-frequency regimes can be combined into the uniform estimate
\[
|\widehat w_{p,q}(\xi)|
\le
C\exp\!\bigl(-c|\xi|^{1/s}\bigr).
\]
Finally, since
\[
\frac1s=\frac{r}{r+1}
=
\left(1+\frac1r\right)^{-1}
=
\left(1+\frac1{\min\{p,q\}}\right)^{-1}
=
\zeta(p,q),
\]
we conclude that
\[
|\widehat w_{p,q}(\xi)|
\le
C\exp\!\bigl(-c|\xi|^{\zeta(p,q)}\bigr).
\]
This completes the proof.
\end{proof}

\begin{theorem}
\label{thm:finite-freq}
Suppose that
\begin{equation}
g_n=a_0+\sum_{\ell=1}^{M} a_\ell e^{2\pi i \nu_\ell n},
\qquad n=0,1,2,\dots,
\label{eq:finite-trig}
\end{equation}
where \(a_0,a_\ell\in\C\), and assume that all nonzero frequencies satisfy
\begin{equation}
\delta:=\min_{1\le \ell\le M}\dist(\nu_\ell,\Z)>0.
\label{eq:gap-dist}
\end{equation}
Then there exist constants \(c,C>0\) such that
\begin{equation}
\bigl|W_N^{(p,q)}(g)-a_0\bigr|
\le
C\exp\!\bigl(-c(\delta N)^{\zeta(p,q)}\bigr),
\qquad
\zeta(p,q)=\left(1+\frac{1}{\min\{p,q\}}\right)^{-1}.
\label{eq:quant-rate}
\end{equation}
\end{theorem}

\begin{proof}
Since \(W_N^{(p,q)}\) is linear and satisfies \(W_N^{(p,q)}(1)=1\), it suffices to estimate
\[
K_N(\nu)
:=
\frac{\sum_{n=0}^{N-1}w_{p,q}(n/N)e^{2\pi i \nu n}}
{\sum_{n=0}^{N-1}w_{p,q}(n/N)}.
\]
Extend \(w_{p,q}\) to a function on \(\R\) by setting it equal to zero outside \((0,1)\). Since \(w_{p,q}(1)=0\), the finite sum may be rewritten as a sum over all lattice points:
\[
\sum_{n=0}^{N-1}w_{p,q}(n/N)e^{2\pi i \nu n}
=
\sum_{n\in\Z}w_{p,q}(n/N)e^{2\pi i \nu n}.
\]
Applying the Poisson summation formula, we obtain
\[
\sum_{n\in\Z}w_{p,q}(n/N)e^{2\pi i \nu n}
=
N\sum_{m\in\Z}\widehat w_{p,q}\!\bigl(N(\nu-m)\bigr).
\]
By Lemma~\ref{lem:fourier-decay},
\[
\left|\sum_{n=0}^{N-1}w_{p,q}(n/N)e^{2\pi i \nu n}\right|
\le
CN\sum_{m\in\Z}\exp\!\bigl(-cN^{\zeta(p,q)}|\nu-m|^{\zeta(p,q)}\bigr).
\]
If \(\dist(\nu,\Z)\ge \delta\), then \(|\nu-m|\ge \delta\) for every \(m\in\Z\). Therefore,
\[
\left|\sum_{n=0}^{N-1}w_{p,q}(n/N)e^{2\pi i \nu n}\right|
\le
C'N\exp\!\bigl(-c'(\delta N)^{\zeta(p,q)}\bigr).
\]

On the other hand, for the normalization factor
\[
A_N^{(p,q)}:=\sum_{n=0}^{N-1}w_{p,q}(n/N),
\]
the Riemann-sum approximation yields
\[
A_N^{(p,q)}=N+o(N)
\qquad (N\to\infty),
\]
since \(\int_0^1 w_{p,q}(x)\,dx=1\). Hence
\[
|K_N(\nu)|
\le
C''\exp\!\bigl(-c''(\delta N)^{\zeta(p,q)}\bigr).
\]
Summing over the finitely many nonzero frequencies in \eqref{eq:finite-trig}, we obtain
\[
\bigl|W_N^{(p,q)}(g)-a_0\bigr|
\le
\sum_{\ell=1}^M |a_\ell|\,|K_N(\nu_\ell)|
\le
C\exp\!\bigl(-c(\delta N)^{\zeta(p,q)}\bigr).
\]
This completes the proof.
\end{proof}

\begin{remark}\label{rem:practical}
Theorem~\ref{thm:finite-freq} provides an asymptotic convergence rate. For finite \(N\), however, the actually observed speed of convergence may also be influenced by the normalization factor \(A_N^{(p,q)}\), the constants arising from higher-order derivatives of the weight, the small-divisor structure, the Fourier coefficients of the observable, and floating-point precision. Therefore, although the theoretical exponent \(\zeta(p,q)\) increases with \(\min\{p,q\}\), the interpretation of numerical plots in concrete experiments should still take into account the structure of the model and the size of the observation window.
\end{remark}

\section{Explicit Integrable Model and Numerical Verification}

\subsection{A Explicit Quantum Integrable Model}\label{subsec:4.1}

On \((\C^2)^{\otimes 3}\), consider the Hamiltonian
\begin{equation}
H=\sum_{j=1}^{3}\omega_j \sigma_j^z,
\qquad
(\omega_1,\omega_2,\omega_3)=(1,\sqrt2,\sqrt3).
\label{eq:H-model}
\end{equation}
Since the operators \(\sigma_j^z\) commute pairwise, this is a non-interacting quantum integrable system.

We choose the initial state
\begin{equation}
\ket{\psi_0} = \ket{+}^{\otimes 3},
\qquad
\ket{+}=\frac{\ket{\uparrow}+\ket{\downarrow}}{\sqrt2},
\label{eq:psi0}
\end{equation}
and define the observable
\begin{equation}
A=\frac13\sum_{j=1}^{3}\sigma_j^x.
\label{eq:A-model}
\end{equation}
For discrete times \(n=0,1,2,\dots\), let
\begin{equation}
\ket{\psi_n}=e^{-inH}\ket{\psi_0},
\qquad
y_n:=\bra{\psi_n}A\ket{\psi_n}.
\label{eq:yn-def}
\end{equation}

\begin{proposition}
\label{prop:explicit-signal}
For the above model, one has
\begin{equation}
y_n=
\frac13\Bigl[\cos(2n)+\cos(2\sqrt2\,n)+\cos(2\sqrt3\,n)\Bigr].
\label{eq:explicit-signal}
\end{equation}
Moreover, the diagonal equilibrium state satisfies
\begin{equation}
\rhodiag=\frac18 I_8,
\qquad
\Aeq:=\Tr(\rhodiag A)=0.
\label{eq:Aeq-zero}
\end{equation}
\end{proposition}

\begin{proof}
For a single spin, one has the standard conjugation formula
\[
e^{in\omega \sigma^z}\sigma^x e^{-in\omega \sigma^z}
=
\cos(2\omega n)\sigma^x+\sin(2\omega n)\sigma^y.
\]
Therefore,
\[
\bra{+}e^{in\omega \sigma^z}\sigma^x e^{-in\omega \sigma^z}\ket{+}
=
\cos(2\omega n)\bra{+}\sigma^x\ket{+}
+ 
\sin(2\omega n)\bra{+}\sigma^y\ket{+}.
\]
Since
\[
\bra{+}\sigma^x\ket{+}=1,\qquad \bra{+}\sigma^y\ket{+}=0,
\]
it follows that
\[
\bra{+}e^{in\omega \sigma^z}\sigma^x e^{-in\omega \sigma^z}\ket{+}
=
\cos(2\omega n).
\]
Combining this with the linearity of \(A\) and the tensor-product structure, we obtain
\[
y_n=\frac13\sum_{j=1}^3 \cos(2\omega_j n),
\]
which is exactly \eqref{eq:explicit-signal}.

On the other hand, \(\ket{+}^{\otimes 3}\) has eight components of equal modulus in the standard \(\sigma^z\)-eigenbasis. After dephasing, only the diagonal entries remain, and hence
\[
\rhodiag=\frac18 I_8.
\]
Moreover, since \(A\) is the average of the operators \(\sigma_j^x\), each of which has zero trace, we have
\[
\Aeq=\Tr\!\left(\frac18 I_8 A\right)=\frac18\Tr(A)=0.
\]
This proves the proposition.
\end{proof}

\begin{corollary}
\label{cor:model-rate}
Define the unweighted average and the weighted average by
\begin{equation}
B_N:=\frac1N\sum_{n=0}^{N-1} y_n,
\qquad
W_N^{(p,q)}
:=
\frac{\sum_{n=0}^{N-1}w_{p,q}(n/N)\,y_n}
{\sum_{n=0}^{N-1}w_{p,q}(n/N)}.
\label{eq:BN-WN}
\end{equation}
Then
\[
W_N^{(p,q)}\to 0
\qquad (N\to\infty),
\]
and there exist constants \(c,C>0\) such that
\begin{equation}
|W_N^{(p,q)}|
\le
C\exp\!\bigl(-cN^{\zeta(p,q)}\bigr),
\qquad
\zeta(p,q)=\left(1+\frac{1}{\min\{p,q\}}\right)^{-1}.
\label{eq:WN-rate}
\end{equation}
\end{corollary}

\begin{proof}
By Proposition~\ref{prop:explicit-signal},
\[
y_n=
\frac16 e^{2in}+\frac16 e^{-2in}
+\frac16 e^{2i\sqrt2\,n}+\frac16 e^{-2i\sqrt2\,n}
+\frac16 e^{2i\sqrt3\,n}+\frac16 e^{-2i\sqrt3\,n}.
\]
This is a finite-frequency trigonometric polynomial with constant term \(a_0=0\). Therefore, Theorem~\ref{thm:finite-freq} applies directly and yields \eqref{eq:WN-rate}. Since \(\Aeq=0\), the limiting value is zero.
\end{proof}

\subsection{Numerical Verification}

The purpose of this part is not to identify the limit numerically, since the limit has already been determined explicitly in the previous section. Rather, our aim is to demonstrate that the weighted averages approach the dephased equilibrium value faster than the ordinary averages, and that, for the model under consideration, larger values of \(\min\{p,q\}\) lead to faster convergence in practice.

For the three-spin model considered in Subsection \ref{subsec:4.1}, we have
\begin{equation}
y_n=
\frac13\Bigl[\cos(2n)+\cos(2\sqrt2\,n)+\cos(2\sqrt3\,n)\Bigr],
\qquad
\Aeq=0.
\label{eq:num-signal}
\end{equation}
We compare the ordinary average
\begin{equation}
B_N=\frac1N\sum_{n=0}^{N-1}y_n
\label{eq:num-BN}
\end{equation}
with the weighted average
\begin{equation}
W_N^{(p,q)}
=
\frac{\sum_{n=0}^{N-1}w_{p,q}(n/N)\,y_n}
{\sum_{n=0}^{N-1}w_{p,q}(n/N)}.
\label{eq:num-WN}
\end{equation}
The parameter pairs are chosen as
\[
(p,q)\in\{(0.5,0.5),(1,1),(2,2),(4,4)\}.
\]
We define the errors by
\begin{equation}
E_N^{\mathrm{unw}}:=|B_N-\Aeq|=|B_N|,
\qquad
E_N^{(p,q)}:=|W_N^{(p,q)}-\Aeq|=|W_N^{(p,q)}|.
\label{eq:num-errors}
\end{equation}
To distinguish more stably the numerical advantage of the case \((4,4)\), the computations were carried out using high-precision arithmetic.

\begin{figure}[H]
    \centering
    \safeincludegraphics[width=0.72\textwidth]{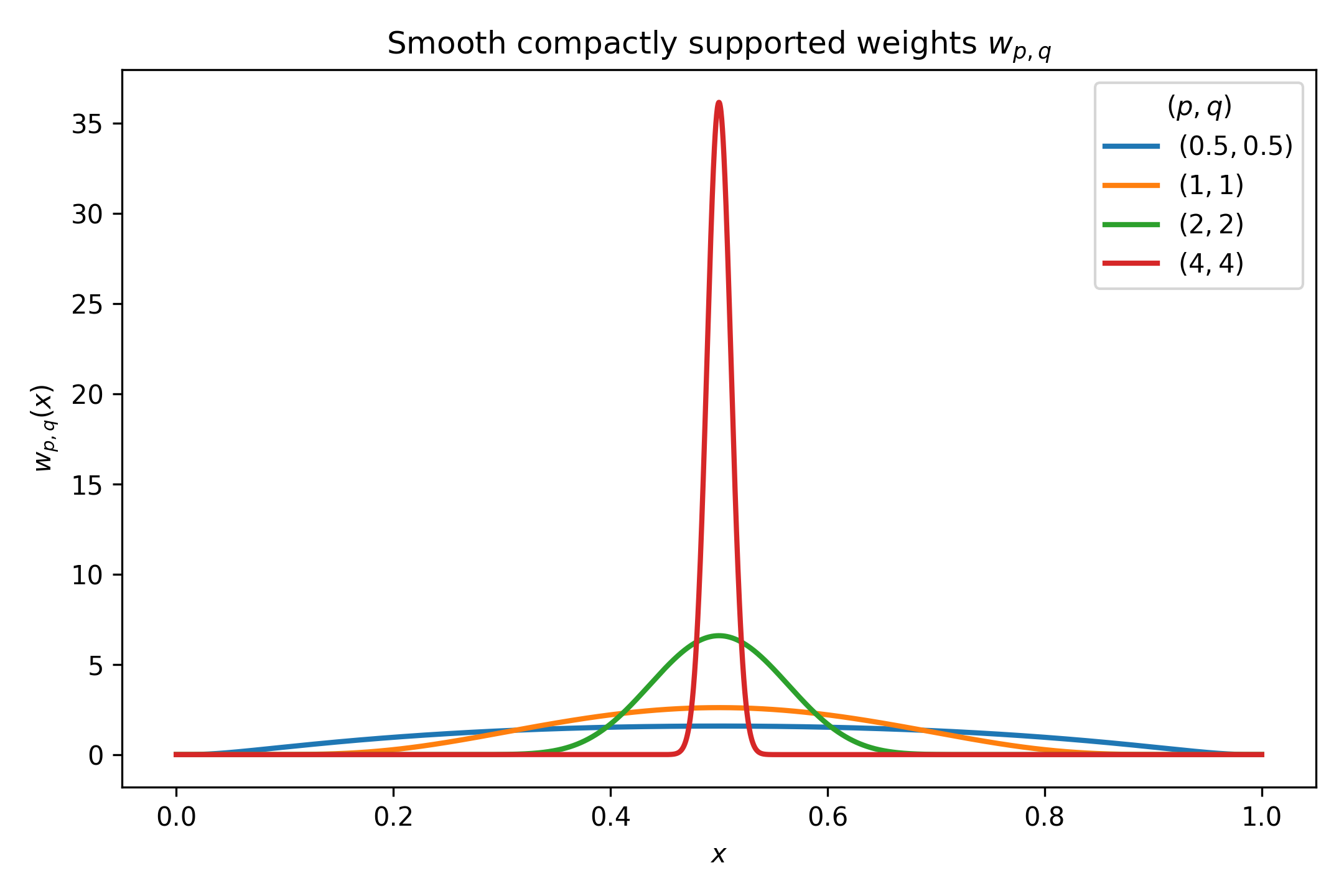}
    \caption{Compactly supported weight functions \(w_{p,q}\) corresponding to the four parameter pairs \((p,q)=(0.5,0.5),(1,1),(2,2),(4,4)\). As \(\min\{p,q\}\) increases, the weight becomes increasingly concentrated in the middle portion of the time window and suppresses the endpoint data more strongly.}
    \label{fig:weight}
\end{figure}

Figure~\ref{fig:weight} displays the weight functions \(w_{p,q}\) corresponding to the four parameter pairs. As \(\min\{p,q\}\) increases, the weight becomes more concentrated in the middle of the time window and suppresses the data near the two endpoints more strongly. This is precisely the basic geometric mechanism by which weighted averaging reduces endpoint errors over a finite observation window.

Figure~\ref{fig:running} illustrates the behavior of the running averages. To make the process of approaching equilibrium more visually transparent, we restrict this figure to the intermediate window \(40\le N\le 400\). In Figure~\ref{fig:running} (when generated by the present code, the two vertically stacked panels are included in a single figure), the upper panel shows the unweighted average together with the four weighted averages themselves, while the lower panel displays the signed deviation from the equilibrium value \(0\) on a symlog vertical axis. One can see that the unweighted average still exhibits a noticeable drift, whereas the weighted averages approach the zero line more rapidly. Moreover, larger values of \((p,q)\) typically enter a smaller deviation scale at an earlier stage.

Figure~\ref{fig:error} presents semilogarithmic decay plots of the errors. Here we focus on the later asymptotic window \(650\le N\le 1200\), in order to distinguish more clearly the hierarchy of convergence rates corresponding to different parameter choices. The figure shows that the weighted averages significantly outperform the unweighted average. Moreover, under the present choice of model, observation window, and high-precision computation, one can clearly observe the overall ordering
\[
E_N^{(4,4)} < E_N^{(2,2)} < E_N^{(1,1)} < E_N^{(0.5,0.5)} < E_N^{\mathrm{unw}}.
\]
This is consistent with the theoretical conclusion of Theorem~\ref{thm:finite-freq}, namely, that the exponent \(\zeta(p,q)\) increases with \(\min\{p,q\}\).

\begin{figure}[H]
    \centering
    \safeincludegraphics[width=0.82\textwidth]{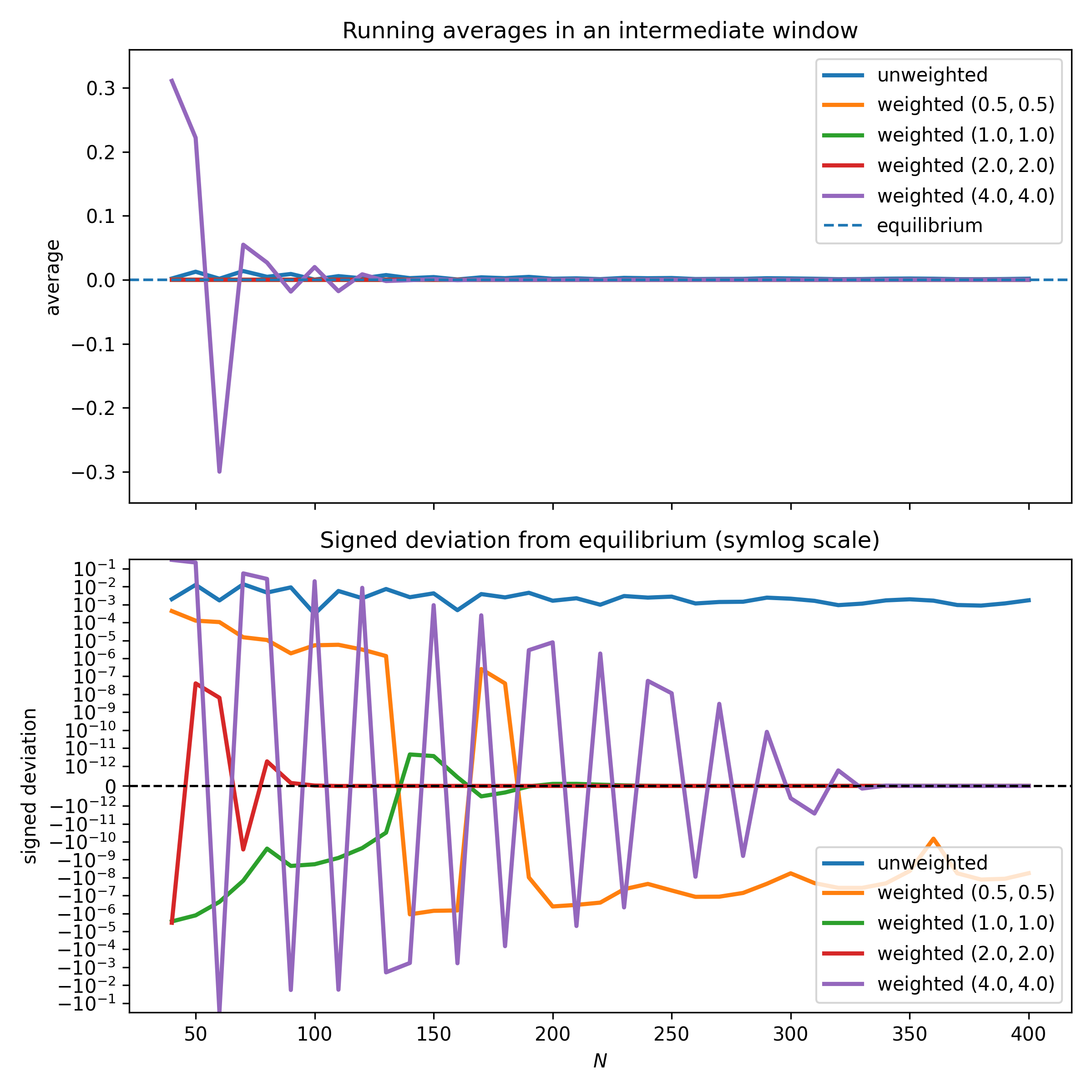}
    \caption{Numerical behavior of the running averages. The upper panel shows the unweighted average and the four weighted averages over the intermediate window \(40\le N\le 400\); the lower panel displays the signed deviation from the equilibrium value \(0\) on a symlog vertical axis. The purpose of this figure is to highlight the \emph{process} of approaching equilibrium, rather than the final asymptotic error.}
    \label{fig:running}
\end{figure}

Finally, Figure~\ref{fig:stretched} provides a linearized test of the stretched-exponential decay law. For each parameter pair, we plot
\[
\log_{10} E_N^{(p,q)}
\quad \text{against} \quad
N^{\zeta(p,q)},
\]
where
\[
\zeta(0.5,0.5)=\frac13,\qquad
\zeta(1,1)=\frac12,\qquad
\zeta(2,2)=\frac23,\qquad
\zeta(4,4)=\frac45.
\]
If the error satisfies
\[
E_N^{(p,q)}\approx \exp(-cN^{\zeta(p,q)}),
\]
then the corresponding plot should be approximately linear. The numerical results show that such a linear trend is clearly visible, thereby supporting the stretched-exponential convergence law established in this paper.

\begin{figure}[H]
    \centering
    \safeincludegraphics[width=0.75\textwidth]{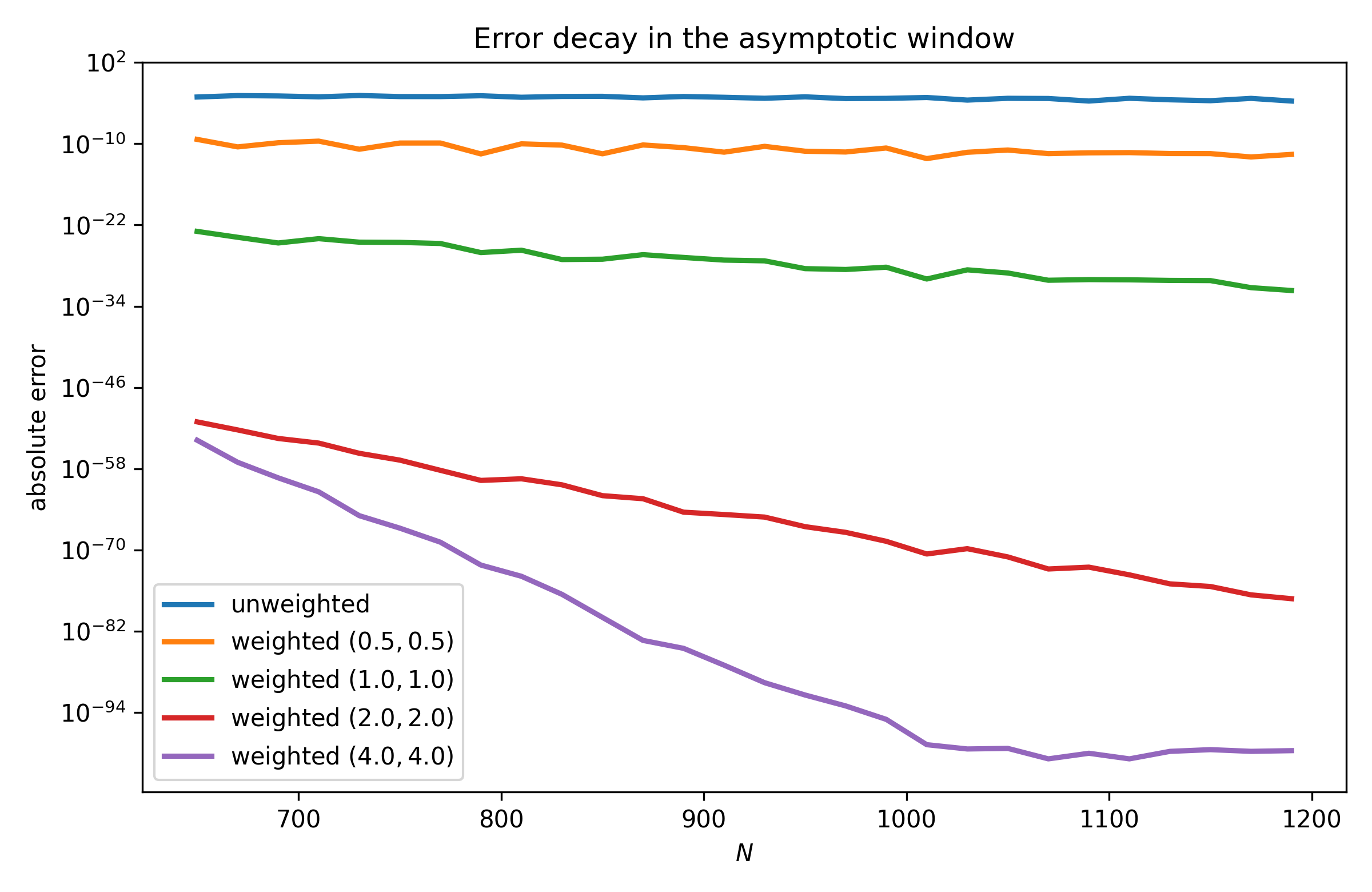}
    \caption{Semilogarithmic decay plots of the errors \(E_N^{\mathrm{unw}}\) and \(E_N^{(p,q)}\) over the asymptotic window \(650\le N\le 1200\). The weighted averages clearly outperform the unweighted average; for the present model, larger values of \(\min\{p,q\}\) correspond to faster error decay.}
    \label{fig:error}
\end{figure}

\begin{figure}[H]
    \centering
    \safeincludegraphics[width=0.75\textwidth]{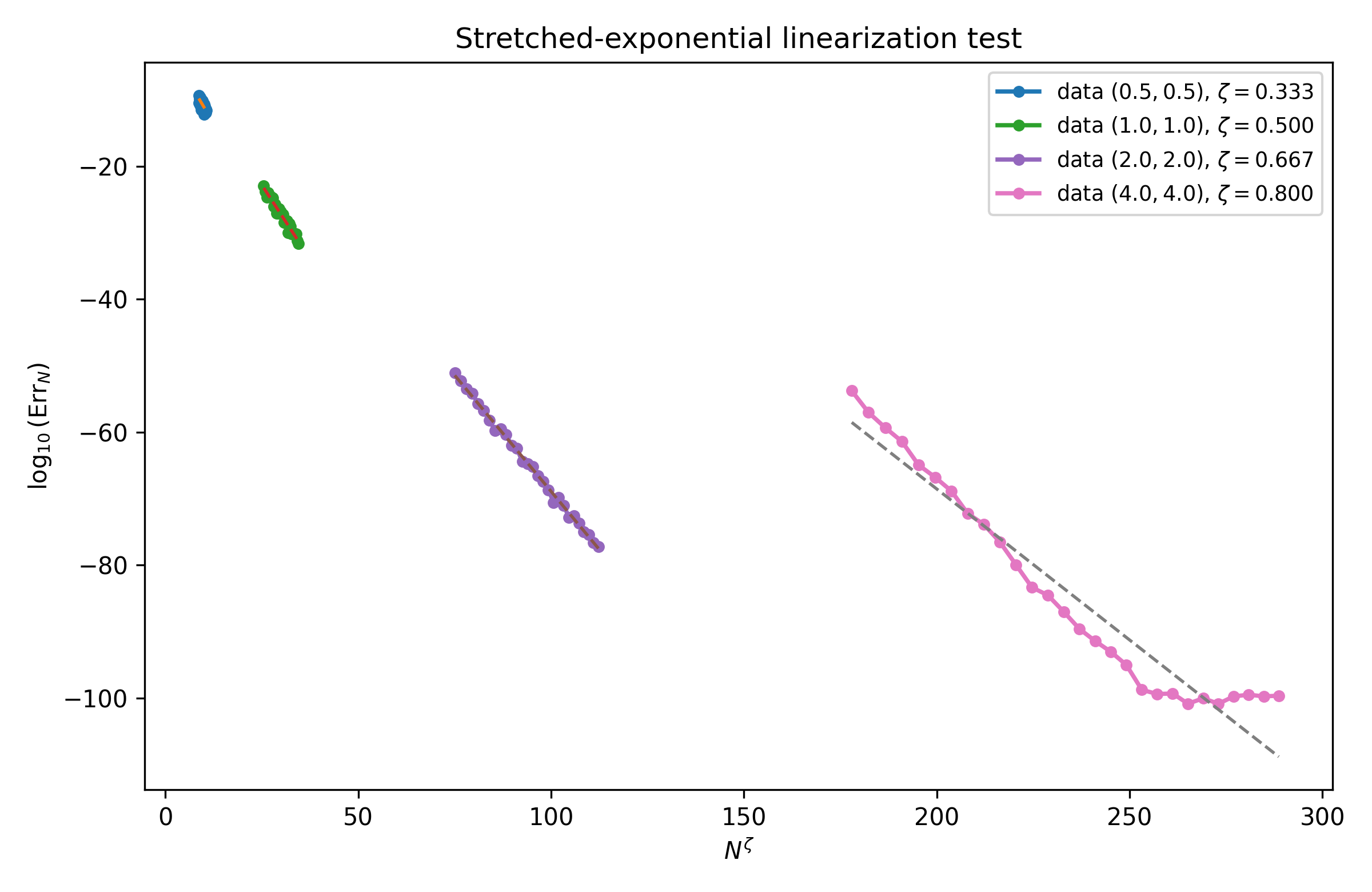}
    \caption{Linearized test of the stretched-exponential convergence law. For each pair \((p,q)\), we plot \(\log_{10}E_N^{(p,q)}\) against \(N^{\zeta(p,q)}\), where \(\zeta(p,q)=(1+1/\min\{p,q\})^{-1}\). The data points exhibit an approximately linear trend overall, thereby supporting the theoretical prediction \(E_N^{(p,q)}\approx \exp(-cN^{\zeta(p,q)})\).}
    \label{fig:stretched}
\end{figure}

In summary, this three-spin quantum integrable model provides a particularly transparent numerical example: single-time observables remain quasiperiodically oscillatory and do not converge pointwise to equilibrium, whereas their weighted time averages rapidly approach the dephased equilibrium value. Moreover, under the chosen observation window and high-precision implementation, larger values of \(\min\{p,q\}\) correspond to faster convergence.

%This section will present the weighted time-averaging procedure, the definition of the diagonal equilibrium state, and the main weak convergence theorem.

\section{Conclusion}\label{sec5}

In this paper, we introduced a weighted time-averaging framework for quantum systems with pure point spectrum and used it to formulate a natural quantum counterpart of weak convergence to equilibrium in integrable dynamics. Since the unitary evolution in this setting typically produces quasiperiodic expectation values, one cannot in general expect pointwise convergence of single-time observables as $t\to\infty$. The appropriate asymptotic object is therefore not a pointwise limit, but a long-time averaged equilibrium state. Our analysis shows that this role is naturally played by the diagonal (dephased) state, which removes oscillatory couplings between distinct energy subspaces while preserving the block structure associated with spectral degeneracies.

At the general level, we proved that the continuous weighted time average of the expectation of any bounded observable satisfying a suitable absolute summability condition converges to the expectation taken with respect to the dephased equilibrium state. We also reformulated this result at the level of averaged density operators and obtained weak convergence of the weighted averaged states on the corresponding class of observables. In this way, the paper establishes a rigorous weak equilibration principle for the class of quantum systems considered here.

To address the convergence speed, we then studied quasiperiodic signals with finitely many frequencies and derived an explicit quantitative estimate for the discrete weighted averages. More precisely, we showed that the convergence is of stretched-exponential type, with exponent
\[
\zeta(p,q)=\left(1+\frac{1}{\min\{p,q\}}\right)^{-1}.
\]
This shows that the weighted averaging procedure does not merely recover the equilibrium value, but also provides a genuine acceleration mechanism relative to ordinary averaging. In particular, the dependence of the exponent on $\min\{p,q\}$ makes precise how stronger suppression of endpoint contributions improves the averaging effect.

Finally, we applied the general theory to an explicit three-spin quantum integrable model. In this example, the observable dynamics can be computed exactly, the dephased equilibrium state is identified explicitly, and the numerical simulations are in clear agreement with the theoretical predictions. Although the single-time signal remains quasiperiodic and does not converge pointwise, its weighted averages approach the equilibrium value much faster than the ordinary averages. For the observation windows considered in this paper, larger values of $\min\{p,q\}$ also lead to better numerical performance, in accordance with the theoretical rate.

Overall, the results of this paper show that weighted time averages provide a natural and effective framework for describing equilibration in quantum systems with pure point spectrum. They connect oscillatory microscopic dynamics with a well-defined ensemble-level equilibrium object and, at the same time, yield a quantitative improvement over standard averaging methods. 

Possible directions for future research include weakening the summability assumptions, extending the analysis to more general almost-periodic frequency structures and infinite-dimensional models, and studying how the present dephasing mechanism interacts with nearly integrable perturbations.

%We have formulated a weighted time-average approach for quantum integrable systems with pure point spectrum and identified the diagonal ensemble as the natural equilibrium object. The detailed proofs and model computations will be developed in the subsequent version of this manuscript.

%\section*{Acknowledgements}

%Not applicable.

\section*{Declarations}

%\textbf{Funding}

%Not applicable.

%\textbf{Conflict of interest}

The author declares that there is no conflict of interest.

\section*{Acknowledgements}
The author Yong Li was supported by National Natural Science Foundation of China
(12071175, 12471183 and 12531009).

%\textbf{Ethics approval and consent to participate}

%Not applicable.

%\textbf{Consent for publication}

%Not applicable.

%\textbf{Data availability}

%No datasets were generated or analysed in this study.

%\textbf{Materials availability}

%Not applicable.

%\textbf{Code availability}

%Not applicable.

%\textbf{Author contribution}

%Xinyu Liu conceived the study, carried out the analysis, and wrote the manuscript.

% Temporary: bibliography is disabled until actual references are added and
% the Springer Nature bibliography style file (for example,
% sn-mathphys-num.bst) is available in the project.

\end{document}